\documentclass[10pt]{amsart}
\usepackage[all]{xy}
\usepackage{amsmath}
\usepackage{amssymb}
\usepackage{amsthm}
\usepackage{verbatim}

\newtheorem{theorem}{Theorem}[section]
\newtheorem{lemma}[theorem]{Lemma}
\newtheorem{definition}[theorem]{Definition}
\newtheorem{remark}{Remark}
\numberwithin{equation}{section}

\numberwithin{equation}{section}
\newtheorem{proposition}[theorem]{Proposition}

\newfont{\EUL}{eufm10 scaled 1000}

\newcommand\R{\mathbb{R}}

\newcommand\C{\mathbb{C}}

\renewcommand\P{\mathbb{P}}

\renewcommand\H{\mathbb{H}}

\newcommand\E{{\rm E}}
\newcommand\G{{\rm G}}

\newcommand\so{\mathfrak{so}}

\newcommand\su{\mathfrak{su}}
\renewcommand\u{\mbox{\EUL u}}

\renewcommand\sp{\mathfrak{sp}}

\newcommand\g{\mathfrak{g}}
\newcommand\cc{\mathfrak{c}}
\newcommand\h{\mathfrak{h}}

\renewcommand\u{\mathfrak{u}}

\newcommand\rk{{\rm rk}\,}

\newcommand\spann{{\rm span}}

\newcommand\SO{{\rm SO}}
\renewcommand\O{{\rm O}}
\newcommand\Spin{{\rm Spin}}
\newcommand\SU{{\rm SU}}

\newcommand\Sp{{\rm Sp}}

\renewcommand\S{{\rm S}}

\newcommand\U{{\rm U}}

\newcommand\Irr{{\rm Irr}}
\newcommand\Ea{{\rm E}}
\newcommand\Fa{{\rm F}}

\newcommand\Jcal{{\mathcal J}}

\begin{document}
%
%
\title[Totally complex submanifolds]{Maximal totally
  complex submanifolds\\ of $\H\P^n$: homogeneity and normal
  holonomy}
\author{Lucio Bedulli - Anna Gori - Fabio Podest\`a}
\address{Dipartimento di Matematica per le Decisioni
- Universit\`a di Firenze\\
via C. Lombroso 6/17\\50134 Firenze\\Italy}
\email{bedulli@math.unifi.it}
\address{Dipartimento di Matematica - Universit\`a di Bologna\\
Piazza di Porta S.\ Donato 5\\40126 Bologna\\Italy}
\email{gori@math.unifi.it}
\address{Dipartimento di Matematica e Applicazioni per l'Architettura
- Universit\`a di Firenze\\
Piazza Ghiberti 27\\50122 Firenze\\Italy}
\email{podesta@math.unifi.it}
\date{}
\begin{abstract} We prove that a maximal totally complex submanifold $N^{2n}$ of the quaternionic projective space
$\H\P^n$ ($n\geq 2$) is a parallel submanifold, provided one of the following conditions is satisfied:
(1) $N$ is the orbit of a compact Lie group of isometries, (2) the
restricted normal holonomy is a proper subgroup of
$\U(n)$. \end{abstract}
\maketitle

\section{Introduction}
A $4n$-dimensional ($n\geqslant2$) quaternionic-K\"ahler Riemannian manifold ($M,g$) (qK-manifold for
brevity) is endowed with a parallel
rank $3$ subbundle $Q$ of ${\rm End}\,TM$, which is locally generated by a triple of locally defined anticommuting
$g$-orthogonal almost complex structures $(J_1,J_2,J_3=J_1J_2)$ (see \cite{Sal}, \cite{Besse} and also the next section
for a more detailed exposition). Here we will focus on positive qk-manifolds $M$, i.e. compact
qk-manifolds with
positive scalar curvature, and on
{\em totally complex} submanifolds of $M$, i.e. submanifolds
$N$ of $M$ for which there exists a local section $I$ of $Q$ which is a
local parallel complex structure on $N$;
the submanifold $N$ equipped with the induced metric and the local complex structure
$I$ turns out to be locally K\"ahler and this is actually equivalent to the fact that $J(TN)\subset TN^\perp$ for every
local complex structure $J\in \Gamma(Q)$ which is orthogonal to $I$ (see \cite{AlMa}, \cite{Ta}). When the dimension of
$N$ is half the dimension of the ambient space, the totally complex submanifold $N$ is called maximal, MTC for brevity;
Tsukada (\cite{parallel}) classified all MTC submanifolds of the quaternionic projective space which are parallel,
i.e. whose second fundamental form is parallel - in particular all
these examples turn out to be homogeneous and locally symmetric (see
also \cite{ADM}); we list
all these submanifolds in Theorem \ref{Tsu}.\\
The first result of this paper consists in the classification of {\em homogeneous} maximal totally complex submanifolds of
the quaternionic projective space, namely
\begin{theorem}\label{main} A compact MTC submanifold of the quaternionic projective space
$\H\P^n$ is parallel if and only if it is an orbit of a compact Lie group of isometries.
\end{theorem}
In section 3 we will also prove some general facts about homogeneous MTC submanifolds $N$ of positive qK-manifolds $M$;
we show that such manifolds admit two different homogeneous lifts in the corresponding twistor space $Z$: one lift,
called the complex lift $\widehat N$, is a homogeneous complex submanifold of $Z$ and covers $N$ two to one; the other
lift, denoted by $\mathcal L(N)$, is a homogeneous Lagrangian
submanifold of $Z$ which is a circle bundle on $N$.
These lifts were introduced by Alekseevsky and Marchiafava
(\cite{AlMatwistor}) and by Ejiri and Tsukada (\cite{Ejiri}) respectively; here
we show their homogeneity whenever $N$ is homogeneous and will use this result in the proof of Theorem \ref{main}.\\
The second main result of this paper concerns the study of the
(restricted) normal holonomy group of a MTC submanifold of $\H\P^n$;
the normal holonomy group is the holonomy group of the normal
connection $\nabla^\perp$ and it turns out to be a powerful tool in the study of
the geometry of submanifolds in general (see e.g.\cite{OL}).
Normal holonomy groups have been investigated for submanifolds
of space forms and also for complex submanifolds of the complex projective space (see \cite{AlDi}, \cite{CDO}).
In the sequel we will denote by $\operatorname{Hol}_p^0(\nabla^\perp)
\subseteq \O(T_pN^\perp)$ the restricted holonomy group of
$\nabla^\perp$ at the point $p\in N$.
Our main result is the following
\begin{theorem}
\label{NH}
Let $N$ be a connected MTC submanifold of $\H\P^n$.\\
1) If $N$ is non parallel then $\operatorname{Hol}_p^0(\nabla^\perp)=\U(T_pN^\perp)\cong\U(n)$;\\
2) if $N$ is parallel (hence a symmetric space of the form $G/K$) then
$\operatorname{Hol}_p^0(\nabla^\perp)=\nu(K^0)$ where $\nu:K \to
\O(T_pN^\perp)$ is the slice representation.
\end{theorem}
We note here that this theorem resembles the recent result by Console, Di Scala and Olmos (\cite{CDO}) according to which
a full, complete complex submanifold $N$ of $\C\P^n$ whose normal holonomy group is not transitive on the unit sphere
of its normal space (hence is not $\U(k)$ where $k = n - \dim_\C N$) is parallel.\\
{\bf Notation.} Lie groups and their Lie algebras will be indicated by
capital and gothic letters respectively.
Moreover the notation for the fundamental weights $\Lambda_i$ of the
simple Lie algebras follows the standard conventions as in \cite{OV},
and we often refer to irreducible representations through the corresponding maximal weights.
Finally if $G$ is a Lie group then $\rk G$ will denote its rank.

\section{Preliminaries}\label{1}

Let $(M,g)$ be a $4n$-dimensional ($n\geqslant2$)
Riemannian manifold and $\nabla$ its Levi-Civita connection.
A quaternionic-K\"ahler structure on $M$ is a $\nabla$-parallel rank 3 subbundle $Q$ of ${\rm End}\,TM$,
which is {\em locally} generated by a triple of locally defined anticommuting $g$-orthogonal almost complex
structures $(J_1,J_2,J_3=J_1J_2)$, a so called {\em local quaternionic basis}. The manifold $M$ endowed with the metric $g$ and the quaternionic-K\"ahler
structure $Q$ will be called a quaternionic-K\"ahler manifold, a qK-manifold for brevity.\\
We recall (see e.g. \cite{Besse}) that a qK-manifold is automatically
Einstein; here we will consider only {\em
positive} qK-manifolds, namely those which are compact with positive scalar curvature.\\
Since the bundle $Q$ is parallel, there are three locally defined
$1$-forms $\alpha_i$ ($i=1,2,3$) such that for every cyclic
permutation $(i,j,k)$ of $(1,2,3)$
$$\nabla J_i = \alpha_k\otimes J_j - \alpha_j\otimes J_k\eqno (1.1)$$
and
$$d\alpha_i + \alpha_j \wedge \alpha_k  =  -\frac{\tau}{4n(n+2)}\omega_i,\eqno (1.2)$$
where $\omega_\ell$ denotes the fundamental 2-form defined by $\omega_\ell(\cdot,\cdot)=g(J_\ell\,\cdot,\cdot)$ relative to $J_\ell$ for $\ell=1,2,3$, $\tau$ is the scalar curvature.\par
An important tool in the quaternionic-K\"ahler geometry is the {\em
twistor space} $Z$ of a non-flat qK manifold
$(M,g,Q)$: if we endow $Q$ with the natural Euclidean structure such that every local
quaternionic basis is orthonormal, then $Z$ is defined to be the sphere bundle associated to $Q$; the twistor
space $Z$ endowed with the natural projection $\pi\colon Z \to M$
is then a $S^2$-bundle over $M$ whose fibers consist of complex structures. It is well known that the twistor space of
a positive qK-manifold is a complex manifold with complex structure $\Jcal$ and it admits a K\"ahler-Einstein metric $\tilde g$ such that the projection $\pi$ is a
Riemannian submersion (\cite{Sal},\cite{Besse}): every fiber (a so
called {\em twistor line}) is a complex submanifold biholomorphic to $\C\P^1$ and
at any point $z \in Z$ the tangent space
$T_zZ$ splits as the $\tilde g$-orthogonal sum of the horizontal
$\mathcal{H}_z \cong T_{\pi(z)}M$ and the vertical subspace
$\mathcal{V}_z: = \ker {\pi_*}_z \cong \C$; then by definition $\Jcal|_{\mathcal{H}_z}$ is identified
with $z$
itself via $\pi_*$ , while  $\Jcal|_{\mathcal{V}_z}$ is the standard complex structure of $\C\P^1$.
We will now consider totally complex submanifolds of a qK-manifold, namely
\begin{definition}\label{defloc}
A submanifold $N$ of a qK-manifold $(M,g,Q)$ is said to
be {\it totally complex} if for every $x\in N$ there exists an open
neighborhhood $U$ of $x$ in $N$ and a section $I$ of $Q_{|U}$ such that
\begin{enumerate}
\item $I(TU)=TU$;
\item $I^2=-{\rm id}$;
\item $J(T_yN)\subseteq (T_yN)^\perp$ for every $y\in U$ and $J \in
  Q_y$ such that $I_yJ=-JI_y$.
\end{enumerate}
\end{definition}
As it was shown by Tsukada in \cite{parallel} the local $g_{|N}-$orthogonal
almost complex structure $I$ on $N$ turns out to be integrable, hence
$(U,g_{|U},I_{|U})$ is a K\"ahler manifold (here the condition $\dim_\R N\geq 4$ is essential, see
\cite{parallel} p. 7). Furthermore if $N$ is a
totally complex submanifold of $M$,
then around any point of $N$ it is possible to choose a local
quaternionic basis $J_1,J_2,J_3$ such that $J_1$ is the local complex
structure of Definition \ref{defloc}. As in
\cite{AlMa} we call such a triple a {\em local adapted basis}.
Note that $J_l(TN)\subseteq TN^\perp$ for $l=2,3$.\\
Totally complex submanifolds of maximal dimension of $\H\P^n$ with parallel second
fundamental form have been classified in \cite{parallel} by
Tsukada, who also characterizes in \cite{KE}  those
whose local K\"ahler structure is Einstein or reducible.
\begin{theorem}{\bf (Tsukada\ \cite{parallel},\cite{KE})}\label{Tsu}
Let $\iota: N \to \H\P^n$ be a totally complex immersion of  a
$2n$-dimensional K\"ahler manifold $N$.
(a)\ If $\iota$ has parallel
second fundamental form, it is locally congruent to one of the following
immersions
\begin{enumerate}
\item$\C\P^n\,\rightarrow \H\P^n$;
\item$\C\P^1\times\SO(n+1)/\SO(2)\times
  \SO(n-1)\,\rightarrow \H\P^n$, $(n\geq4)$;
\item$\C\P^1\times\C\P^1\times\C\P^1\,\rightarrow \H\P^3$;
\item$\C\P^1\times\C\P^1 \,\rightarrow \H\P^2$;
\item $\Sp(3)/\U(3)\,\rightarrow \H\P^6$;
\item$\SU(6)/\S(\U(3)\times \U(3))\,\rightarrow \H\P^9$;
\item$\SO(12)/\U(6)\,\rightarrow \H\P^{15}$;
\item$\E_7/E_6\cdot T^1\,\rightarrow \H\P^{27}$;

\end{enumerate}
(b)\ if $N$ is K\"ahler-Einstein then it has parallel second fundamental
form, and $\iota$ is locally congruent to one of the cases $(1),(3),(5),(6), (7), (8)$; \\
(c)\ if $N$ is locally reducible, then $\iota$ is again parallel and it is locally congruent
to one of the cases $(2),(3),(4)$.
\end{theorem}
\begin{remark}{\em All of the immersions listed in Theorem \ref{Tsu}
    are obtained in the following way. For each compact simple Lie group $S$ it is
    known (see e.g. \cite{Besse}) that there exists a unique symmetric space $S/L$ (called
    a Wolf space) admitting an invariant qk-structure. Given a Wolf space $S/L$, the isotropy $L$ can be decomposed
    as $L = \Sp(1)\cdot G$,
    where $G$ is compact semisimple and the $\Sp(1)$-factor acts on $V:= T_{[eL]}S/L$ defining a $G$-invariant
    qk-structure; therefore the isotropy representation of $G$ results into a quaternionic representation 
    $\rho:G\to \Sp(V)$. We then have a
    $G$-equivariant twistor fibration  $\pi:\P_\C(V)\to\P_\H(V)$  and a parallel
    totally complex submanifold of $\P_\H(V)$ is the image via $\pi$
    of a complex $G$-orbit in $\P_\C(V)$ (this orbit is unique if
    $\rho$ is irreducible); in the following table the parallel immersions listed in Theorem \ref{Tsu} part (a) are
    reproduced in the same order with the indication of the corresponding Wolf space, the group $G$ (or, with abuse 
    of notation, some covering of it)
    and its isotropy representation.} 
\end{remark}
\begin{small}
\begin{table}[h]
\centering
\begin{tabular}{|l|l|l|l|l|l|}
\hline \quad\quad\qquad S/L &$\quad\quad G$   & $\quad \rho$                              &  $\dim_\C \P(V)$ \\
\hline $\SU(n+2)/\S(\U(2)\times\U(n))$ &$\U (n)$   & $\Lambda_1\oplus \Lambda_1^*$           &  $2n-1$        \\
       $\SO(n+6)/\SO(4)\times\SO(n+2)$ &$\SO(n+2)\times\SU(2)$ & $\Lambda_1\otimes \Lambda_1$ & $2n+3, n\geq 3$\\
       $\SO(8)/\SO(4)\times\SO(4)$ &$\SU(2)\times\SU(2)\times \SU(2)$ & $\Lambda_1\otimes\Lambda_1\otimes\Lambda_1$ & $7$   \\
       $\SO(7)/\SO(4)\times \SO(3)$ &$\SU(2)\times\SU(2)$ &  $(2\Lambda_1)\otimes\Lambda_1$ & $5$   \\
       $\Fa_4/\Sp(1)\cdot\Sp(3)$ &$\Sp (3)$    &  $\Lambda_3$                     &  $13$         \\
       $\Ea_6/\Sp(1)\cdot\SU(6)$ &$\SU (6)$    &  $\Lambda_3$                     &  $19$   \\
       $\Ea_7/\Sp(1)\cdot\Spin(12)$&$\Spin (12)$ &  $\Lambda_5$                      &  $31$         \\
       $\Ea_8/\Sp(1)\cdot\Ea_7$ &$\Ea_7$      &  $\Lambda_1$                      &  $55$      \\

\hline
\end{tabular} \vspace{0.5cm}
\label{lista}
\end{table}
\end{small}
Note that the Wolf space $\Sp(n+1)/\Sp(1)\cdot\Sp(n)$ is missing in the
list because it gives the transitive standard action of $\Sp(n)$ on
$\P_\H(V)$; the Wolf space $\G_2/\SO(4)$ is also missing
because we consider quaternionic projective spaces of real dimension
at least 8.\\ 

Using the twistor
construction one can settle two kinds of correspondence involving
totally complex submanifolds.  Let $N$ be a totally complex
submanifold of $M$.
For every $p\in N$ consider the set $\widehat{N}_p:=\{I\in Z_p|\,I(T_pN)=T_pN
\}$. The space $\widehat{N}:=\cup_{p\in N}\widehat{N}_p\subset
Z$ turns out to be a complex submanifold (see
\cite{Ta},\cite{parallel} and \cite{AlMatwistor}) of the twistor space
$Z$ which is horizontal, i.e. at each point $z\in \widehat N$ we
have $T_z\widehat N \subset \mathcal{H}_z$.
The restriction $\pi|_{\widehat N}$ defines a double covering of $N$
which may be either connected or not. We will call $\widehat N$ the {\it complex lift\/}
of $N$.\\
The second correspondence makes use of the symplectic
structure of $Z$ instead of the complex one. Fix $p \in N$ and define
$S_p(I)=\{J\in Z_p : J(T_pN)\subseteq (T_pN)^\perp\}$; it is
not difficult to see that $\mathcal{L}(N)=\cup_{p\in N}S_p(I)$ is a connected
isotropic submanifold of $Z$ which is an $S^1$-bundle over $N$
(see \cite{Ejiri} where it is also shown that $\mathcal{L}(N)$ is
minimal). We will call $\mathcal{L}(N)$ the {\em isotropic lift} of $N$.

\section{Homogeneous totally complex submanifolds}

We will now suppose to have a compact connected Lie group $G$ of isometries of a positive qK-manifold ($M,g,Q$) and we
assume that $G$ has an orbit $N = G\cdot p$ which is a totally complex submanifold.
Now we are going to present a sort of equivariant version of the previous constructions.
It is well known that we can lift the $G$-action to an isometric action on the twistor
space $Z$ defining
\[
g\cdot J=g_*Jg_*^{-1}
\]
for every $J \in Z$ and $g \in G$, so that the projection $\pi$ is $G$-equivariant (see e.g. \cite{Besse}). The complex
lift $\widehat{N}$ turns out to be $G$-stable by definition. Every
$G$-orbit in $\widehat{N}$ is open (because it has the same dimension
of $\widehat{N}$) and closed by compactness; therefore every connected
component of $\widehat{N}$ is $G$-homogeneous.
This fact allows us to make the following observation.
\begin{lemma}\label{centralizer}
The semisimple part $G^{ss}$ of $G$ acts transitively on $N$. In particular if $N\cong G^{ss}/H$ for some subgroup $H$, then the
connected component $H^o$ is the centralizer of some torus in $G^{ss}$.
\end{lemma}
\begin{proof}
Suppose by contradiction that $N$ contains a toric factor $S$ (which
is in turn totally complex). Note that, since the twistor space $Z$ is
simply connected,
the $G$-action on $Z$ is Hamiltonian; therefore the complex lift $\hat{S}$
would be simultaneously complex and isotropic in the twistor space
$Z$, being the orbit for
the Hamiltonian action of the center of $G$. The last assertion follows form the fact that $G^{ss}/H^o$ can be identified
with a connected component of $\widehat N$ that is a K\"ahler homogeneous space (see e.g. \cite{Besse}).
\end{proof}
So from now on we can restrict our attention to {\em semisimple} Lie groups.
\begin{remark}{\rm The parallel immersions listed in Theorem \ref{Tsu} (a) that are coming
    from irreducible representations (all cases except (1)) are not
    embeddings. Actually the immersion $\iota$ is $G$-equivariant and
    the image $\iota(N)$ is an orbit $G/H$ in $\H\P^n$, where $G$ acts
    irreducibly on $\C^{2n+2}$: if $p\in \iota(N)$ with $G_p = H$, we
    know that there is a point $q$ in the twistor line lying over $p$ such that the orbit
$\widehat{\iota(N)} = G\cdot q \subset \C\P^{2n+1}$ is complex.
The point $q$ corresponds to a complex structure $I$ in $Z_p$
that leaves $T_p \iota (N)$ invariant and the same argument can be
used with the opposite complex structure, say $q'$, leading to
another complex orbit $G\cdot q'$ lying over $\iota(N)$.
If $G$ acts irreducibly on $\C^{2n+2}$ there is only one complex orbit
in the projective space, hence $G\cdot q = G\cdot q' =
\widehat{\iota(N)}$ and the covering $G/G_q\to G/H$ is two to one.
Since $G_q$ is connected, $H$ has two connected components;
so, for instance, if $N=\Sp(3)/\U(3)$ then
$\iota(N)=\Sp(3)/\U(3)\cdot {\mathbb Z}_2\hookrightarrow \H\P^6$.
This also shows that the local complex structure $I$ in the
definition \ref{defloc} cannot be chosen to be global in general.}\end{remark}
The fact that also $\mathcal{L}(N)$ is a $G$-orbit requires a more
detailed explanation and will have deeper consequences. By the very definition
of the induced actions of $G$ on $Z$, we see that $\mathcal{L}(N)$ is $G$-stable.
\begin{proposition}
\label{isoorbit}
The lift $\mathcal{L}(N)$ is an isotropic $G$-orbit in
$Z$ of dimension $\dim_\R N + 1$. In particular if $N$ is a {\em maximal}
totally complex $G$-orbit, then $\mathcal{L}(N)$ is a {\em Lagrangian}
$G$-orbit.
\end{proposition}
\begin{proof} We fix $p\in N$. We know there is a suitable neighborhood $U$ of $p$ and a
local section $I$ of $Z|_U$ such that $I(TU) = TU$; we also fix $J\in Z_p$, a complex structure
that is orthogonal to $I_p$ and $J(T_pN)\subseteq (T_pN)^\perp$. \\
We first observe that $G_p\cdot I_p=\pm I_p$, simply because $\pm I_p$ are the only complex structures in $Z_p$ that
preserve $T_pN$. Now we know that the $G$-orbit $G\cdot J$ is contained in $\mathcal L(N)$; we claim that
$\dim G\cdot J = \dim N + 1$, hence that $G\cdot J = \mathcal L(N)$. It is enough to show that $\dim G_p - \dim G_J = 1$
or equivalently that $\dim G_p\cdot J = 1$.
It is clear that $\dim G_p\cdot J \leq 1$ because the $G_p$ action on the
twistor line $Z_p$ is not transitive since $G_p\cdot I_p = \pm I_p$.\\
Suppose now that $(G_p)^o\cdot J = J$: this means that the orbit $G\cdot J$ covers at most two to one $N$ via
$\pi$ and therefore we can find a local section $\tilde J$ of
$\mathcal L(N)$, regarded as a $S^1$-bundle over $N$, extending $J$ such that
${\textit L}_X \tilde{J}= 0$ for every Killing vector field $X\in \g$; then using $\nabla I|_N = 0$ and
equation (1.1), we see that $\nabla \tilde{J} = \alpha\otimes I\tilde{J}$
for some locally defined $1$-form $\alpha$. This local one form $\alpha$ satisfies ${\textit L}_X \alpha= 0$ for
every $X\in \g$, hence in particular it extends to an invariant one form on the homogeneous space $G^{ss}/(G^{ss}_p)^o$.
Since $G^{ss}_p$ has maximal rank in $G^{ss}$ by Lemma \ref{centralizer}, we see that $\alpha = 0$; this contradicts equation (1.2), since on $N$ we would have $0 = d\alpha =-\frac{\tau}{4n(n+2)} \omega$,
where $\omega$ is the fundamental two-form on $N$ relative to
$I$.
\end{proof}

\begin{remark}
{\rm It is shown in \cite{CAG} that the existence of a Lagrangian $G$-orbit in a K\"ahler manifold $Z$
with $h^{1,1}(Z)=1$ is equivalent to the fact that the complexified group
$G^\C$ acts on $Z$ with a open Stein orbit. Thus, for example, in
order to classify the MTC orbits of a Wolf
space $K/H\cdot \Sp(1)$, since the twistor space satisfies the
cohomological condition, one can use the following strategy: Look for
the closed subgroups $G$ of $K$ whose complexification has a open Stein
orbit in $Z=K/H\cdot T^1$; find the Lagrangian orbit $\mathcal{L}$
using moment map techniques as suggested e.g. in \cite{CAG} and
project it down in the Wolf space $K/H\cdot \Sp(1)$.   \\
This procedure is easier in the case of the quaternion projective
space since here the twistor space is the complex projective space and
there are various classification results on complex {\em
  pre-homogeneous vector spaces} (see e.g. \cite{SK}).
If $G$ is simple one could in principle use the classification of Lagrangian
orbits of $\C\P^n$ given in \cite{CAG}.}
\end{remark}

\section{Homogeneous MTC submanifolds of $\H\P^n$: proof of
Theorem \ref{main}}
Let $G$ be a compact connected Lie group of isometries acting on
$\H\P^{n}=Sp(n+1)/\Sp(n)\cdot\Sp(1)$ ($n\geqslant 2$) with a MTC orbit $N$.
Such an action is induced by a complex $(2n+2)$-dimensional linear
representation of quaternionic type $\rho: G \to \Sp(V)$ where
$V\cong\C^{2n+2}\cong\H^{n+1}$ is a quaternionic vector space.
As already observed we can restrict our investigation to compact semisimple
Lie groups.\\
The proof will be divided into cases.
\subsection{Case 1: $\rho$ irreducible and $G$ simple}
In order to have a MTC $G$-orbit
the following dimensional condition must be satisfied:
\begin{equation} \label{dim1}\dim G-\rk G\geqslant
  \frac{1}{2}\dim_\R\P_\H(V)=\frac{1}{2}(\dim_\R V -4)=\dim_\C V-2.
\end{equation}
If $\rk G=1$ then $G=\SU(2)$, so that $\dim_\C V=4$ and
we will treat this case separately.
If $\rk G\geqslant 2$, then (\ref{dim1}) implies that $\dim_\C V \leq \dim G$ and
we can use the well known classification of
such irreducible representations (see
e.g. \cite{Fu} p. 414). Going through this list, we easily see that
the representations of quaternionic type that satisfy condition (\ref{dim1})
are the following:
\begin{table}[h]

\begin{tabular}{|c|l|}
\hline  $G$      &  $\rho$\\
\hline
\hline  $\SU(2)$ &  $\Lambda_1$  \\
\hline  $\SU(6)$ & $\Lambda_3$  \\
\hline  $\Sp(3)$  & $\Lambda_3$ \\
\hline  $\Spin(11)$  & $\Lambda_4$ \\
\hline  $\Spin(12)$  & $\Lambda_5$ \\
\hline  $\E_7$  & $\Lambda_1$ \\
\hline

\end{tabular}\medskip
\caption{$G$ simple and $\rho$ irreducible}
\end{table}\par
Note that the actions of $\Spin(11)$ and $\Spin(12)$ on $\C\P^{31}$
via spin and half-spin representations respectively share the same
orbits. This can be easily seen noting that the half-spin representation of
$\Spin(12)$ extends the spin representation of $\Spin(11)\subset
\Spin(12)$ and checking that these actions share the same cohomogeneity.
\subsection{Case 2: $\rho$ irreducible and $G$ non simple}

Fix $p \in N$ and a maximal torus $T$ which fixes $p$ ; we decompose
$V=\bigoplus_{\mu \in \Lambda} V_\mu$, where $\Lambda$ is the set of
weights relative to $T$ and $V_\mu$ is the weight space relative to
$\mu\in\Lambda$; we denote by $\lambda$ the highest weight.\\
The complex orbit $\widehat{N}$ in the twistor space
$Z=\C\P^{2n+1}=\P_\C(V)$ is the orbit through the class of a highest weight vector
$v_\lambda$. Thus the tangent space to $Z$ at $z:=[v_\lambda]$, as a
$T$-module, is $(\C v_\lambda)^*\otimes (\C v_\lambda)^\perp$, hence it has
$\{-2\lambda\}\cup\{-\lambda + \mu,\ \mu\neq\pm\lambda\}$ as weights.
On the other hand, since $\widehat{N}$ is horizontal, we have the
$T$-invariant splitting
\begin{equation}
\label{split}
V_{-\lambda}\otimes(\bigoplus_{\mu\neq\lambda}V_\mu) \cong T_z Z=\mathcal{V}_z \oplus \, T_z\,Gz \oplus \,
((T_z\,Gz)^\perp \cap \mathcal{H}_z).
\end{equation}
where $\pi_*$ gives a $G_p$-isomorphism of $T_z\,G z$ with
$T_p\,Gp$ and of $(T_z\,G z)^\perp \cap \mathcal{H}_p$ with
$J(T_p\,Gp)$, where $J$ is a fixed element of $Z_p$ anticommuting with $z$.\\
Let us see how $T$ operates on the three modules in the right-hand side of \eqref{split}. On
$\mathcal{V}_z$ the weight is $-2\lambda$, since the twistor line
$Z_p$ is the projectivization of
$\spann_\C\{v_\lambda,v_{-\lambda}\}$. The weights of $T$ acting on
$T_zGz$, endowed with the $\mathcal{J}_z$-complex structure, are roots
by the standard theory of invariant complex structures on generalized complex
manifolds (see e.g.\cite{Besse}).
As for the third module we have to compute the action of $T$ on $J$:
the torus $T$ acts on $Z_p\cong S^2$ fixing $z$ with infinitesimal action on
$T_zZ_p$ given by the weight $-2\lambda$; it follows that $T$ acts by
rotations on the set of complex structures orthogonal to $I_p$ by the
same weight $-2\lambda$, i.e. $[H,J]=2i\lambda(H)IJ$.
Therefore if $v$ is a weight vector in the complex module $T_zGz$
relative to a root $\alpha$, we have
\begin{eqnarray*}
HJv & = & JHv + [H,J]\,v \, =\, J(\alpha(H)v)+2i \lambda(H)IJv\\
           & = & -(\alpha(H)+2\lambda(H))Jv,
\end{eqnarray*}
where we have used the fact that every root is $i\mathbb R$-valued on $\mathfrak t$.\\
Thus the weights of $T$ acting on $(T_z\,G z)^\perp \cap \mathcal{H}_p$ are of the form $-2\lambda-\alpha$ where $\alpha$ is a root.
We now compare the weights of $T$ acting on $T_zZ$ using \eqref{split}: for every
weight $\mu\in \Lambda$ with $\mu\neq\lambda,-\lambda$, we have that
$\mu-\lambda$ is either a root
or is of the form $-2\lambda-\alpha$
where $\alpha$ is a root; we deduce that every weight
$\mu\neq\lambda,-\lambda$ is of the form
\begin{equation}\label{pesi}
\mu=\alpha\pm\lambda
\end{equation}
where $\alpha$ is a root of $\mathfrak{g}^\C$.\par
We now split $G=G_1\times\cdots\times G_s$ and accordingly
$V$ as the tensor product $V_1\otimes \cdots \otimes V_s$ where
$V_i$ is of quaternionic type for $i=1,\ldots,2r+1$
and of real type whenever $i=2r+2,\ldots,s$.
Now we have the following
\begin{lemma} In the above situation
\begin{itemize}
\item[1.] For every $i=1,\ldots,2r+1$ we have $V_i \cong \C^2$ and $G_i = \SU(2)$;
\item[2.]$G$ has at most three simple factors, i.e. $s\leq 3$;
\item[3.]If $s=3$, then $G_i=\SU(2)$ and it acts on $V_i\cong\C^2$ via the standard representation for every $i=1,2,3$.
\end{itemize}
\end{lemma}
\begin{proof}
1. We fix $1\leq k\leq s$ and put $\widetilde{G}_k = \prod_{j\neq
  k}G_j$ and $\widetilde{V}_k=\bigotimes_{j\neq i}V_j$.
 Denote by
$\lambda_k$ and $\widetilde{\lambda}_k$ the relevant highest weights, so
that the highest weight of $V$ will be $\lambda_k+\widetilde{\lambda}_k$.
Recall that every weight of $V$ is of the form $\mu+\nu$ where
$\mu$ is a weight of $V_k$ and $\nu$ is a weight of $\widetilde{V}_k$,
whilst a root of $\mathfrak{g}^\C$ is either a root of $\mathfrak{g}^\C_k$
or a root of $\widetilde{\mathfrak{g}}^\C_k = \bigoplus_{j\neq k}\mathfrak{g}^\C_j$. \\
Suppose now that $\dim_\C \widetilde V_k \geq 3$; then there is
a weight $\nu_0\neq\pm\widetilde{\lambda}$ of $\widetilde{V}$; then
\eqref{pesi} becomes
\begin{equation}
\label{pesi2}
\mu+\nu_0=\pm(\lambda_k+\widetilde{\lambda}_k) + \; {\rm root} \; {\rm
  of} \; \mathfrak{g}^\C
\end{equation}
for arbitrary $\mu$ and it implies that either $\mu=\lambda_k$ or
$\mu=-\lambda_k$, i.e. $V_k \cong \C^2$ and for $1,\ldots,2r+1$, $G_i = \SU(2)$ acts on
$V_i$ in the standard way.\\ In particular, this holds when $k\leq 2r+1$ since $\widetilde V_k$ is
 irreducible and of real type, hence its dimension is at least three. \\
2. Using the same notations as above, we put $k=1$ and use \eqref{pesi2} with $\mu=-\lambda_1$ and $\mu=\lambda_1$;
we get  $\nu_0= -\widetilde{\lambda}_1+\alpha= \widetilde{\lambda}_1+\beta$
where $\alpha$ and $\beta$ are two roots of
$\widetilde{\mathfrak{g}}^\C$.
Thus $\widetilde{\lambda}_1=\frac{1}{2}(\alpha-\beta)$ so that
$\widetilde{G}_1$ has at most two simple factors.\\
3. If $s=3$, for every $1\leq k\leq 3$, we have $\dim_\C\widetilde V_k\geq 4$ and therefore, by
the results in (1), $G_k=\SU(2)$ and $V_k=\C^2$.\end{proof}

Therefore we are left with the case where $G = \SU(2)\times G'$, $V=\C^2\otimes V'$ and the group $G'$ is
simple, acting linearly on $V'$ via an irreducible representation $\rho'$ of real type.

To treat this case note that the dimensional
condition \eqref{dim1} becomes
\begin{equation} \label{dim2}\dim G'-\rk G'\geqslant
  \frac{1}{2}(\dim_\R V'\otimes {\C}^2)-2=2\dim_\C V'-2.
\end{equation}
Thus if $\rk G'\geqslant 2$, equation (\ref{dim2}) shows that $\dim
G'\geq 2\dim_\C V'$; hence using the above mentioned list of
\cite[p. 414]{Fu}, one sees that $(G',\rho')$ is forced to be
isomorphic to $(\SO(n+2),\Lambda_1)$ with
$n\geqslant 3$,
$(\Spin(7),\Lambda_3)$ or to $(\G_2,\Lambda_1)$; if $\rk G'=1$,
then $\dim_\C V_1=3$ and we get that $\rho'$ is the adjoint
representation of $\SU(2)$.
We collect these cases in the following table:
\begin{table}[h]
\label{nsg}
\begin{tabular}{|c|l|}
\hline                    $G$\   & \quad $\rho$\\
\hline
\hline          $\SO(n+2)\times\SU(2)$ &  $\Lambda_1\otimes\Lambda_1$  \\
\hline          $\Spin(7)\times\SU(2)$ &  $\mbox{spin}\otimes\Lambda_1$  \\
\hline              $\G_2\times\SU(2)$ &  $\Lambda_1\otimes\Lambda_1$  \\
\hline            $\SU(2)\times\SU(2)$ &  $(2\Lambda_1)\otimes\Lambda_1$  \\
\hline
\end{tabular}\medskip
\caption{Non simple groups}
\end{table}\\
All the cases in Table 3 give rise to
MTC orbits which turn out to be parallel by Theorem \ref{Tsu}, (1).\\
Note that the group $\G_2\times\SU(2)$ has the
same totally complex orbit of $\SO(7)\times\SU(2)$ and this holds also
for  $\Spin(7)\times\SU(2)$ with $\SO(8)\times\SU(2)$: this follows from the fact that
the groups $\G_2$ and $\Spin(7)$ act transitively on the quadrics $\SO(7)/\SO(2)\times\SO(5)$ and
$\SO(8)/\SO(2)\times\SO(6)$ respectively (see e.g. \cite{O}).

\subsection{Case 3: $\rho$ reducible.}
Let $V$ be a reducible $G$-module with an invariant quaternionic structure, given by an invariant anti-linear
map $q$ with $q^2 = -{\rm id}$. We then decompose $V$ as
\begin{equation}\label{V}V = \bigoplus_{i=1}^k V_i \oplus \bigoplus_{i=1}^s(W_i \oplus W_i^*),
\end{equation}
where $V_i$ are $q$-stable irreducible $G$-modules, while $W_i$ are
$G$-irreducible and $q(W_i)=W_i^*$;
here we allow the indices $k$ and $s$ to be zero, meaning
that there are no summands of the corresponding type.\\
We now fix $p\in N \subset \P_\H(V)$ and $[v]\in {\widehat N} \subset \P_\C(V)$ a point of the complex lift of
$N$ lying over $p$. Since $G\cdot [v]$ is complex, we see
that $v$ can be written (up to a reordering) as
\begin{equation}
\label{dec}
v = \sum_{i=1}^hv_i + \sum _{i=1}^t (w_i + w_i^*),\quad 0\neq v_i\in V_i\quad
0 \neq w_i + w_i^*\in W_i+W_i^*,
\end{equation}
where $h\leq k$, $t\leq s$; every component in the decomposition above
is a highest weight vector (w.r.t. the Cartan subalgebra
$\frak t^\C$ and a suitable choice of positive roots $R^+$):
indeed, if $G\cdot[v] = G^\C/P$ where $P$ is a parabolic subgroup whose Lie algebra
 contains $\frak t^\C$ and all positive root spaces $\g_\alpha$ for $\alpha\in R^+$, then $\frak t^\C\cdot v\in \C\cdot v$
and $\g_\alpha\cdot v = 0$ for all $\alpha\in R^+$. Moreover, since $\frak t^\C\cdot v\in \C\cdot v$, we see that all these highest weights are
equal and we have the following
\begin{lemma} (1)\ $V_1\cong\ldots\cong V_h$; \quad (2)\ for every $1\leq i,j\leq t$ either $W_i\cong W_j$ or
$W_i\cong W_j^*$;\quad (3)\ if $h\geq 1$ and $t\geq 1$, then $W_1\cong\ldots\cong W_t\cong V_1$
\end{lemma}
\begin{proof} The first two claims follow from the discussion above, while (3) is a consequence of the fact that
$V_1$ is of quaternionic type, hence self-conjugate.\end{proof}
We can therefore rewrite the decomposition \eqref{V} as
\begin{equation} V = kV_1 \oplus s(W\oplus W^*),\end{equation}
for some nonnegative integers $k,s$ and some irreducible $G$-modules $W$ and $V_1$ (of quaternionic type).
We now prove the following
\begin{lemma} We have $h=k$ and $t=s$.
\end{lemma}
\begin{proof} Indeed suppose $h<k$ or $t<s$; this implies that the quaternionic span of $v$ is contained in a proper
$G$-invariant subspace $U$ of $V$. On the other hand by Proposition \ref{isoorbit} we know that there is a point $w$ in the quaternionic span of $v$ such that
the orbit $G\cdot[w]$ is Lagrangian; this means that the $G^\C$-orbit through $[w]$ is open in $\P_\C(V)$, while it
is contained in $\P_\C(U)$, a contradiction.\end{proof}
\begin{lemma} \label{final}One of the following holds :
\begin{enumerate}
\item \label{quat} $k=2$, $s=0$;\ $G = \Sp(m)$ and $V_1\cong \C^{2m}$ is the standard representation;
\item \label{cplx} $k=0$ and $s=1$;\ $G=\SU(m)$ or $G=\Sp(m)$ and $W$ is the standard
representation $W\cong \C^m$ or $W\cong \C^{2m}$ respectively.
\end{enumerate}
\end{lemma}
\begin{proof} Suppose $k\geq 1$. This means that the orbit $N$ is diffeomorphic to $G/K$ , where $K$ is the
stabilizer in $G$ of $[v_1]$. Since $N$ is MTC, we have
\begin{equation}\label{est}
\dim_\R \P_\C(V_1) = 2\dim_\C V_1 - 2 \geq \dim_\R G/K = \frac{1}{2}\dim_\R\P_\H(V)=\dim_\C V - 2,
\end{equation}
hence $\dim_\C V\leq 2\dim_\C V_1$. This means that $k\leq 2$. Suppose $k=2$: then $s=0$ and we have
the equality in \eqref{est}; this means that $G$ acts transitively on $\P_\C(V_1)$, hence $G=\Sp(n)$ because
$V_1$ is of quaternionic type (see e.g. \cite{O}). If $k=1$ and $s\geq 1$, then
$W_i\cong V_1$ and $\dim_\C V \geq 3\dim_\C V_1$, a contradiction.  \\
If $s\geq 1$, we repeat the same estimate as in \eqref{est} using $W_1$ (or $W_1^*$, according to whether $v$ has a
non zero component along $W_1$ or $W_1^*$) and we get that $s=1$ and $k=0$. Again we have that $G$ acts transitively
on the projective space $\P_\C(W)$ and the claim follows.\end{proof}
We now note that the cases listed in Lemma \ref{final} are actually possible and
correspond to $\C\P^{m-1}=\SU(m)/\S(\U(1)\times\U(m-1))\hookrightarrow \H\P^{m-1}$ and $\C\P^{2m-1}=\Sp(m)/T^1\times
\Sp(m-1)\hookrightarrow \H\P^{2m-1}$ ; observe that the $\Sp(m)$-modules $V=2\C^{2m}$, endowed with the quaternionic structure
that restricts to the standard quaternionic structure of each summand $\C^{2m}$, and $V'=\C^{2m}\oplus (\C^{2m})^*\cong
2\C^{2m}$ endowed with the quaternionic structure that interchanges the two summands, are $\Sp(m)$-equivariantly
isomorphic.

\section{Normal holonomy}
In this section we compute the normal holonomy of a MTC submanifold $N$ of $\H\P^n$.\\
Given a point $p\in N$ we choose a local
adapted basis $\{J_1:=I,\,J_2:= J,J_3\}$, i.e. sections over a suitable neighborhood $U$ of $p$ of the bundle
$Z$, such that $I$ induces a complex structure along $N$ and $J$ is a section of $\mathcal{L}(N)$ over $U$.\\
We denote by $\nabla^\perp$ the normal connection in the
normal bundle $TN^\perp$ and by $R^\perp$ its
curvature tensor. \\
Here we remark that the submanifold $N$ and all the connections
involved are real analytic, thus the relevant holonomy algebras are
generated by the endomorphisms given by the covariant derivatives of
the curvature tensors (see e.g. \cite{KN}, p.153).
In particular this implies that the restricted holonomy group can be
determined via local computations.\\
From now on we will work in the open subset $U$;
since $J$ can be seen as an isomorphism between
the bundles $TU$ and $TU^\perp$, we can pull back the normal
connection to $TU$ defining
\begin{equation}\label{def}
\nabla'_X (Y):=J^{-1}\nabla_X^\perp(J\,Y).
\end{equation}
Obviously ${\rm Hol}_p^0(\nabla^\perp)\cong{\rm Hol}_p^0(\nabla')$; we
will denote by $\h$ and $\h'$ the Lie algebras of ${\rm
  Hol}_p^0(\nabla)$ and ${\rm Hol}_p^0(\nabla')$ respectively. It is easy to show that (see e.g. \cite{KE})
\begin{equation}\label{conn}\nabla'_X Y = \nabla_XY -\alpha(X) IY,\end{equation}
where $\nabla$ is the Levi Civita connection of the induced metric on $N$ and $\alpha$ is a local one form such that
$\nabla J = \alpha\otimes IJ$ (see equation 1.1); moreover
\begin{equation}\label{curv}
R'_{XY}Z = R_{XY}Z + \frac{\tau}{4n(n+2)}\omega(X,Y)\ IZ,
\end{equation}
where $R'$ is the curvature tensor of the connection $\nabla'$ and $\omega$ is the local K\"ahler form of $I$.
We now observe that, using \eqref{conn} and the K\"ahler condition $\nabla I = 0$ on $U$,
$$\nabla\omega = \nabla'\omega = 0,$$
so that we can compute the covariant derivatives
$$\nabla'R'_{XY} = \nabla'R_{XY}$$
and again using \eqref{conn} we get
\begin{equation}\label{dercurv}
\nabla'R'_{XY} = \nabla R_{XY}\end{equation}
for every $X,Y\in T_pN$.
It follows from \eqref{curv}, \eqref{dercurv} and the real analyticity that
$\operatorname{Span}\{\h',I_p\} = \operatorname{Span}\{\h,I_p\}$.

\begin{lemma}
\label{hol2} $\h'$ contains a nontrivial center $\cc$.
\end{lemma}
\begin{proof}
Our claim will follow from the fact that $\h' \not\subseteq\su(T_pN)$: indeed this
implies that $\h' \neq [\h',\h']$ and therefore $\h'$, being a
subalgebra of $\so(T_pN)$, contains a nontrivial center $\cc$.

Suppose that $\h' \subseteq \su(T_pN)$; then
by \eqref{curv} we have (here $\rho$ denotes the Ricci form of
the induced K\"ahler metric on the adapted neighborhood $U$)
\begin{equation}
\label{Einsteinconstant}
0 = \operatorname{Trace}(I\circ R'_{XY}) = 2\rho(X,Y) - \frac{\tau}{2(n+2)}\omega(X,Y),
\end{equation}
hence $U$ is Einstein with constant scalar cuvature $\frac{n}{2(n+2)}\tau$.\\
Therefore $N$ is parallel and is locally
congruent to one of the cases (1), (3), (5), (6), (7) or (8) listed in Theorem
\ref{Tsu}.
We now observe that the only locally reducible case is (3): in this case $\h$
is abelian and therefore $\h'$ is also abelian by \eqref{curv}.
As for the remaining cases (1), (5)--(8), in order to get a contradiction, we need to
inspect more carefully the immersion $N \to \H\P^n$.
In case (1) $\C\P^n \hookrightarrow \H\P^n$ has scalar curvature given
by $\frac{n+1}{4(n+2)}\tau$, a contradiction since $n>1$.
As for the cases (5)--(8) note that the twistor projection $\pi$ restricted to the complex
lift $\hat{N}$ is a local isometry. By Theorem \ref{Tsu}
$\hat{N}$ is a Einstein homogeneous complex submanifold of
$\C\P^{2n+1}(c)$, where $c=\frac{\tau}{4n(n+2)}$; therefore
the result of Takeuchi \cite[Theorem 4.3]{Ta2} gives that the
scalar curvature of such immersions is $\frac{n}{6(n+2)}\tau$ (see
also \cite[Proposition 3.12]{AlMa}).
Comparing with \eqref{Einsteinconstant} we obtain the desired contradiction.
\end{proof}

{\it Proof of Theorem \ref{NH}, part (1)}.
We first prove that $\h=\u(T_pN)$ and therefore it contains $I_p$.
Indeed $N$ is locally irreducible by Theorem \ref{Tsu}, (c); thus using Berger's classification
of the holonomy of irreducible Riemannian manifolds, we see that either ($U,g,I$) is
Hermitian symmetric or $\h$ is isomorphic to $\sp(n)$,
$\su(n)$ or $\u(n)$. If it is a symmetric space or the
holonomy is properly contained in $\u(n)$, then ($U,g$) is Einstein
and therefore it is a parallel submanifold by Theorem \ref{Tsu}, (b).
It follows that $\h = \u(T_pN)$.\\
Since $[\cc,I_p]=0$ and $\h'$ together with $I_p$ span the whole
$\h=\u(T_pN)$, we see that $\cc$ centralizes
$\u(T_pN)$, hence $\cc = \mathbb R\cdot I_p$. Therefore $\h'=\h\cong
\u(n)$.\\

{\it Proof of Theorem \ref{NH}, part (2)}.
We represent the parallel MTC submanifold $N$ as a symmetric space
$G/K$. Let $i:K\to \O(T_pN)$ and $\nu:K\to \O(T_pN^\perp)$ be respectively the isotropy
and slice representations at $p$. Since $\hat{N}=G/K^0$ is a Hermitian
symmetric space one has $i(K^0)=\operatorname{Hol}^0_p(\nabla)$.
\begin{lemma}
$\operatorname{Hol}^0_p(\nabla^\perp)$ is contained in $\nu(K)$.
\end{lemma}
\begin{proof}
The key fact is that the parallel submanifold $N$ has curvature
invariant normal space (see e.g. \cite{Fun} or \cite{AlMa}), hence, by a result of
Naitoh \cite{Nai}, it is a {\em symmetric submanifold}. Thus we can
apply the same arguments as in \cite[Proposition 2.4]{CD} to reach the conclusion.
\end{proof}

Now observe that if we prove that $\h=\h'$ then we have $\dim \h'=\dim
K$ and
\[
\dim K=\dim \h'=\dim \operatorname{Hol}^0_p(\nabla^\perp) \leq \dim
\nu(K) \leq \dim K
\]
so that $\operatorname{Hol}^0_p(\nabla^\perp)=\nu(K)$.
The following lemma concludes the proof.
\begin{lemma}
$I_p \in \h'$, hence $\h=\h'$.
\end{lemma}
\begin{proof}
Let us distinguish two cases.\\
If $N$ is irreducible the center $\cc$ of $\h'$ is $\R\cdot I_p$
because $\h$ acts irreducibly on $T_pN$.\\
In the remaining cases (2), (3) and (4) of Theorem \ref{Tsu} we may
use \ref{curv} and the explicit expression for the curvature of the
involved symmetric spaces;
It is not difficult to exhibit $X$ such that
$R_{XIX}=\pm I_p$ and obtain the claim.
\end{proof}

\end{document}